\documentclass[11pt,reqno]{amsart}
\usepackage{amssymb,amsfonts,amscd,amsbsy,color}
\usepackage[mathscr]{eucal}
\usepackage{url}

\newtheorem{thm}{Theorem}[section]
\newtheorem{lem}[thm]{Lemma}
\newtheorem{prop}[thm]{Proposition}

\theoremstyle{definition}
\newtheorem{remark}[thm]{Remark}

\newcommand{\la}{\lambda}
\newcommand{\spt}{\operatorname{spt}}
\newcommand{\ospt}{\operatorname{ospt}}
\newcommand{\oospt}{\overline{\operatorname{ospt}}}

\numberwithin{equation}{section}

\makeatletter
\def\imod#1{\allowbreak\mkern5mu({\operator@font mod}\,\,#1)}
\makeatother

\begin{document}

\title[The number of strings along the overpartitions]
{On the number of even and odd strings along the overpartitions of $n$}

\author{Byungchan Kim }
\address{School of Liberal Arts \\ Seoul National University of Science and Technology \\ 232 Gongneung-ro, Nowongu, Seoul,139-743, Republic of  Korea}
\email{bkim4@seoultech.ac.kr}

\author{Eunmi Kim}
\address{Center for Applications of Mathematical Principles, National Institute for Mathematical Sciences\\70 Yuseong-daero 1689-gil, Yuseong-gu, Daejeon 305-811, Republic of Korea}
\email{ekim@nims.re.kr}

\author{Jeehyeon Seo}
\address{Department of Mathematics, Sogang University \\ 35 Baekbeom-ro, Mapo-gu, Seoul 121-742, Republic of Korea}
\email{meshgrid80@gmail.com}

\thanks{This research has been supported by TJ Park Science Fellowship of POSCO TJ Park Foundation.}

\date{\today}

\begin{abstract}
Recently, Andrews, Chan, Kim and Osburn introduced the even strings and the odd strings in the overpartitions. We show that their conjecture
\[
A_k (n) \ge B_k (n)
\]
holds for large enough positive integers $n$, where $A_k (n)$ (resp. $B_k (n)$) is the number of odd (resp. even) strings along the overpartitions of $n$. We introduce $m$-strings and show how this new combinatorial object is related with another positivity conjecture of Andrews, Chan, Kim, and Osburn. Finally, we confirm that the positivity conjecture is also true for large enough integers.
\end{abstract}

\maketitle

\section{Introduction}

A partition of a non-negative integer $n$ is a non-increasing sequence of positive integers whose sum is $n$. For example,  The $5$ partitions of 4 are $4$,  $3+1$, $2+2$, $2+1+1$, and $1+1+1+1$. To explain Ramanujan's famous three partition congruences, the partition statistics, the rank (resp. the crank) were introduced by Dyson \cite{dyson} (resp. Andrews and Garvan \cite{ag}). Atkin and Garvan \cite{atg} initiated the study of rank and crank moments to examine the relations between the rank and the crank.  For $k \geq 1$, the $k$-th rank moment $N_{k}(n)$ and the $k$-th crank moment $M_{k}(n)$ are given by
\begin{align*}
N_{k}(n) &:= \sum_{m \in \mathbb{Z}} m^{k} N(m,n)
\intertext{and}
\label{kcrank}
M_{k}(n)&:= \sum_{m \in \mathbb{Z}} m^{k} M(m,n),
\end{align*}
where $N(m,n)$ (resp. $M(m,n)$) denotes the number of partitions of $n$ whose rank (resp. crank) is $m$. By the symmetries $N(-m,n)=N(m,n)$ \cite{dyson} and $M(-m,n)=M(m,n)$ \cite{ag}, we see $N_{k}(n)=M_{k}(n)=0$ for $k$ odd. To define nontrivial odd moments, Andrews, Chan and Kim \cite{ack} introduced the positive rank and crank moments
\begin{align*}
{N}_{k}^{+}(n) &:= \sum_{m=1}^{\infty} m^{k} N(m,n)
\intertext{and}
{M}_{k}^{+}(n) &:= \sum_{m=1}^{\infty} m^{k} M(m,n).
\end{align*}
Andrews, Chan and Kim \cite{ack} showed that
\begin{equation*} \label{oddrc}
{M}_{k}^{+}(n) > {N}_{k}^{+}(n),
\end{equation*}
for all positive integers $k$ and $n$. Indeed, this extended Garvan's result \cite{gar2}:
\begin{equation*} \label{rcforpart}
M_{2k}(n) > N_{2k}(n),
\end{equation*}
for all positive integers $k$ and $n$ (Note that $M_{2k} (n) = 2 M_{2k}^{+} (n)$ and $N_{2k} (n) = 2 N_{2k}^{+} (n)$). As ${M}_{k}^{+}(n) - {N}_{k}^{+}(n)$ is positive, it is natural to ask what they count.  In this direction, Andrews \cite{Andrews:spt} showed that
\begin{equation*}
\label{sptmoment}
\spt(n) = {M}^{+}_2(n) - {N}^{+}_2(n),
\end{equation*}
where $\spt(n)$ is the number of the smallest parts in the partitions of $n$. Moreover, Andrews, Chan and Kim introduced a new counting function $\ospt(n)$ satisfying
\begin{equation*}
\label{osptdefn}
\ospt(n) ={M}^{+}_1(n) - {N}^{+}_1(n),
\end{equation*}
where $\ospt(n)$ counts the number of certain strings along the partitions of $n$ (For the precise definition of $\ospt(n)$, see \cite{ack}). More recently, Andrews, Chan, Kim and Osburn \cite{acko} considered an overpartition analog of positive moments. Recall that an overpartition \cite{lc} is a partition in which the first occurrence of each distinct number may be overlined. For example,  the $14$ overpartitions of $4$ are \begin{equation*}
\begin{gathered}
4, \overline{4}, 3+1, \overline{3} + 1, 3 + \overline{1},
\overline{3} + \overline{1}, 2+2, \overline{2}
+ 2, 2+1+1, \overline{2} + 1 + 1, 2+ \overline{1} + 1, \\
\overline{2} + \overline{1} + 1, 1+1+1+1, \;\text{and}\;\;\overline{1} + 1 + 1 +1.
\end{gathered}
\end{equation*}
Let $\overline{N}(n,m)$ denote the number of overpartitions of $n$ whose rank is $m$ and $\overline{M}(n,m)$ denote the number of overpartitions of $n$ whose (first residual) crank is $m$ \cite{blo}. As an analog of ordinary positive moments, Andrews, Chan, Kim and Osburn \cite{acko} defined positive rank and crank moments for overpartitions:
\begin{align*}
\overline{N}_{k}^{+}(n) &:= \sum_{m=1}^{\infty} m^{k} \overline{N}(m,n)
\intertext{and}
\overline{M}_{k}^{+}(n) &:= \sum_{m=1}^{\infty} m^{k} \overline{M}(m,n).
\end{align*}
As an analog of $\spt(n)$ and $\ospt(n)$, they defined
\[
\oospt(n) := \overline{M}_{1}^{+}(n) - \overline{N}_{1}^{+}(n).
\]
 To explain what $\oospt(n)$ counts, they introduced even and odd strings, which are our main objects. We define an odd string starting from $2k-1$ in an overpartition as
\begin{enumerate}
\item $2k-1, 2k, \ldots, 2k+2\ell-3$ appears at least once, i.e. there are $2\ell -1$ consecutive part sizes starting from $2k-1$.
\item There is no other part of size $2\ell^2 - \ell$ and there is no part of size $4\ell + 2k -2$.
\end{enumerate}
Similarly, we define an even sting starting from $2k$ in an overpartition as
\begin{enumerate}
\item $2k-1, 2k, \ldots, 2k+2\ell-2$ appears at least once, i.e. there are $2\ell$ consecutive part sizes starting from $2k-1$.
\item There is no other part of size $2\ell^2 + \ell$ and there is no part of size  $4\ell + 2k$.
\end{enumerate}
Let  $A_{k} (n)$ be the number of odd strings starting from $2k-1$ along the overpartitions of $n$, and let  $B_{k} (n)$ be the number of even strings starting from $2k-1$ along the overpartitions of $n$. Then, Andrews, Chan, Kim and Osburn showed that
\[
\oospt(n) =   \sum_{k=1}^{\lfloor(n+1)/2\rfloor} ( A_{k}(n) - B_{k} (n) )
\]
and conjectured that $A_k (n) \ge B_k (n)$. Our main result is that this conjecture is true for large enough integers.

\begin{thm}\label{mainthm}
For all positive integers $k$,
\[
A_k (n) > B_k (n)
\]
for all large enough integers $n$.
\end{thm}

 We are going to prove Theorem \ref{mainthm} by employing the circle
method. Due to the linear sum of partial theta functions, the
generating function for $A_k (n) - B_k (n)$ is not modular. Therefore,
we cannot use modular transformation formula to apply the circle method.  However, by
analyzing the behavior of generating function near $q=1$ and away from
$q=1$, we can obtain a desirable asymptotic formula for the
difference $A_k (n) - B_k (n)$. As the difference grows exponentially, we deduce the main
theorem. This approach has been very successful to get asymptotic
formula for the coefficients of $q$-expansions involving partial theta functions, and a similar approach was used in \cite{bm2}.

 Before stating the second result, we need to introduce some notations. In \cite[Lemma 2.2]{acko}, Andrews, Chan, Kim, and Osburn introduced the function
 \begin{equation*} \label{hqdefn}
 h(q) := \sum_{n=1}^{\infty} \frac{(-1)^{n+1} q^{n(n+1)/2}}{1-q^n}  = \sum_{j=1}^{\infty} q^{j^2} (1+2q^{j} + 2 q^{2j}
+ \cdots + 2 q^{j^2 -j} + q^{j^2} ),
 \end{equation*}
 and conjectured that
 \[
 \frac{1}{(q)_{\infty}} ( h(q) - m h(q^m) )
 \]
 has positive coefficients, where we use the standard $q$-series notation, $(a)_{\infty} = (a;q)_{\infty} = \prod_{n=1}^{\infty} (1-aq^{n-1} )$.  Here we define a new combinatorial object to give a combinatorial interpretation for the coefficients of $ \frac{1}{(q)_{\infty}} h( q^m )$.  We define $m$-string in an ordinary partition as the parts consisting of $m(1+k)$, $m(3+k)$,  $\ldots$, $m(2j-1+k)$ with a positive integer $j$ and a nonnegative integer $k \le j$. We also define a weight of $m$-string as 1 if $k=0$ or $j$, and 2, otherwise. Then, we define $C_m (n)$ as the weighted sum of the number of $m$-strings along the partitions of $n$, i.e.
 \[
 C_m (n) = \sum_{\la \vdash n} \sum_{\substack{\text{$\pi$ is a} \\ \text{$m$-string of $\la$}}} \operatorname{wt} (\pi).
 \]
 It is straightforward to see that
 \[
 \frac{1}{(q)_{\infty}} h(q^m) = \sum_{n \ge 1} C_m (n) q^n.
 \]
 Therefore, the second conjecture of Andrews, Chan, Kim, and Osburn \cite{acko} can be translated as there are more (weighted count of) $1$-strings than $m$ times of (weighted count of) $m$-strings along the partitions of $n$.  By employing the similar argument before, we can prove that the above conjecture is also asymptotically true.

\begin{thm}\label{second main}
For all integer $m\geq 2$,
\[
\frac{1}{(q)_{\infty}}(h(q)-mh(q^m))
\]
has positive power series coefficients for all large enough integers $n$, i.e.
\[
C_1 (n) > m C_m (n)
\]
holds for large enough integers $n$.
\end{thm}

\begin{remark} \label{rem2}
Despite the differences are positive for large integers $n$, the proof of Theorem \ref{second main} reveals that, for any positive integer $m \ge 2$,
\[
C_1 (n) \sim m C_m (n)
\]
as $n$ goes to the infinity.
\end{remark}

\begin{remark} \label{rem1}
For all integer $m\geq 2$,
\[
\frac{(-q)_{\infty}}{(q)_{\infty}}(h(q)-mh(q^m))
\]
has also positive power series coefficients for all large enough integers $n$. Actually, by comparing the generating function of positive moments in \cite{acko}, we see that
\[
\overline{C}_1 (n) - 2 \overline{C}_2 ( n) = \overline{M}_{1}^{+} (n) - \overline{N}_{1}^{+} (n) = \oospt(n),
\]
where $\overline{C}_{m}(n)$ be the weighted count of the number of $m$-strings along the overpartitions of $n$. Then, by the exactly same method for the proof of Theorem \ref{second main}, we find
\[
\oospt(n) \sim \frac{1}{64 n} e^{\pi \sqrt{n}}  \sim  \frac{1}{8} \overline{p} (n).
\]
We also note that $\overline{C}_1 (n) \sim m \overline{C}_{m} (n)$ for positive integers $m \ge 2$.
\end{remark}

\section{Proofs of Results}
For a positive integer $a$ and an integer $b$ with $a+b> 0$, we define
\begin{equation*}
	f_{j,a,b}(\tau ):= \sum_{n=1}^{\infty} (-1)^n n^{-j} q^{(an^2+bn)/2}.
\end{equation*}
Here and throughout the paper, we set $q=e^{2\pi i \tau}$ with $\tau = x+iy$. From \cite{acko}, we write the generating function $F_k (q)$ as
\begin{align*}
	F_k(q)&= \sum_{n=1}^{\infty} (A_k(n)-B_k(n))q^n\\
	&= \frac{(-q)_{\infty}}{(q)_{\infty}}\left[2 f_{0,2,(4k-2)}( \tau )- f_{0,1,(4k-1)}(\tau)- f_{0,1,(4k-3)}(\tau)\right].
\end{align*}
We begin with investigating an asymptotic behavior of partial theta function.
\begin{lem}\label{f_0}
	Assume $|x|\leq y$. As $y \rightarrow 0+$,
	\begin{equation*}
		f_{0,a,b}(\tau)=-\frac{1}{2}+\frac{b}{8}(-2 \pi i \tau) + \frac{ab}{32}(-2 \pi i \tau)^2 + \mathcal{O}(y^3)
	\end{equation*}
\end{lem}
\begin{proof}
	If we separate even and odd terms, then we obtain
	\begin{equation*}
		f_{a,b}(\tau) = e^{-\frac{b^2}{4a}\pi i \tau} g_{a,b}(\tau),
	\end{equation*}
	where
	\begin{equation*}
		g_{a,b}(\tau)=\sum_{n=0}^{\infty} \left[ e^{4\pi a \left(n+1+\frac{b}{4a}\right)^2 i \tau} - e^{4\pi a \left(n+\frac{1}{2}+\frac{b}{4a}\right)^2 i \tau} \right].
	\end{equation*}
	The real part of $g_{a,b}(\tau)$ can be written as
	\begin{equation*}
		\text{Re} \left(g_{a,b}(\tau)\right) = \sum_{n=0}^{\infty} \left[u_{\frac{x}{y}}\left( \left(n+1+\frac{b}{4a}\right)\sqrt{y} \right)-u_{\frac{x}{y}}\left( \left(n+\frac{1}{2}+\frac{b}{4a}\right)\sqrt{y} \right)\right],
	\end{equation*}
	where
	\begin{equation*}
		u_s(t)=e^{-4 \pi a t^2} \cos(4\pi a s t^2)=1 - 4 \pi a t^2  + 8 \pi^2 a^2 (1-s^2) t^4 +\mathcal{O}(t^6) ~~\text{as}~~ t \rightarrow 0+.
	\end{equation*}
	By Zagier's result on asymptotic expansions for series (the first generalization of Proposition 3 in \cite{z} with a correction on the sign), for $\dfrac{b}{4a}+\dfrac{1}{2}>0$ (this is the case as $a>0$, $a+b>0$),
	\begin{align*}
		\text{Re} \left(g_{a,b}(\tau)\right)
		=& \left[ \frac{I_u}{\sqrt{y}}-B_1\left(1+\frac{b}{4a}\right)-(- 4 \pi a)\frac{B_3\left(1+\frac{b}{4a}\right)}{3} y \right.\\
		& \hspace{90pt} \left. - 8 \pi^2 a^2 \frac{B_5\left(1+\frac{b}{4a}\right)}{5} (y^2-x^2)+\mathcal{O}_{\frac{x}{y}}(y^3) \right]\\
		& -\left[ \frac{I_u}{\sqrt{y}}-B_1\left(\frac{1}{2}+\frac{b}{4a}\right)-(- 4 \pi a)\frac{B_3\left(\frac{1}{2}+\frac{b}{4a}\right)}{3} y \right.\\
		& \hspace{90pt} \left. - 8 \pi^2 a^2 \frac{B_5\left(\frac{1}{2}+\frac{b}{4a}\right)}{5} (y^2-x^2) + \mathcal{O}_{\frac{x}{y}}(y^3) \right]
	\end{align*}
	where $|I_u|=\left| \int_{0}^{\infty} u_s(t) \, dx \right| < \infty$ from the assumption $|x|\leq y$.
	Similarly, 
	\begin{align*}
		&\text{Im} \left(g_{a,b}(\tau)\right)\\
		&= \left[ \frac{I_v}{\sqrt{y}}-4 \pi a\frac{B_3\left(1+\frac{b}{4a}\right)}{3} x- (-16 \pi^2 a^2) \frac{B_5\left(1+\frac{b}{4a}\right)}{5} xy + \mathcal{O}_{\frac{x}{y}}(y^3) \right]\\
		& \quad -\left[ \frac{I_v}{\sqrt{y}}-4 \pi a\frac{B_3\left(\frac{1}{2}+\frac{b}{4a}\right)}{3} x - (-16 \pi^2 a^2) \frac{B_5\left(\frac{1}{2}+\frac{b}{4a}\right)}{5} xy + \mathcal{O}_{\frac{x}{y}}(y^3) \right]
	\end{align*}
	where $v_s(t)=e^{-4 \pi a t^2} \sin(4\pi a s t^2)$ and $|I_v|=\left| \int_{0}^{\infty} v_s(t) \, dx \right| < \infty$.
	The real and imaginary parts together will give us
	\begin{equation*}
		g_{a,b}(\tau)=-\frac{1}{2}-\frac{b(2a+b)}{8a} (\pi i \tau) +\frac{b(8a^3-4ab^2-b^3)}{64a^2}(\pi i \tau)^2 + \mathcal{O}(y^3).
	\end{equation*}
     By considering the Taylor expansion of $e^{-\frac{b^2}{4a}\pi i \tau}$, we get the claimed expansion. 
\end{proof}

Now we estimate $F_k(q)$ near $q=1$.
\begin{prop}\label{Sum_near1}
	Assume $y=\dfrac{1}{4 \sqrt{N}}$ and $|x|\leq y$. As $N \rightarrow \infty$,
	\begin{equation*}
		F_k(q)=\frac{2k-1}{16\sqrt{\pi}}(-2 \pi i \tau)^{5/2}  e^{\frac{\pi i}{8 \tau}} + \mathcal{O}\left(N^{-7/4}e^{\frac{\pi}{2}\sqrt{N}}\right).
	\end{equation*}
\end{prop}
\begin{proof}
	 From Lemma \ref{f_0}, we conclude that 
	\begin{equation*}
		2 f_{0,2,(4k-2)}(\tau)- f_{0,1,(4k-1)}(\tau)- f_{0,1,(4k-3)}(\tau) = \frac{2k-1}{8}(-2 \pi i \tau)^2 + \mathcal{O}(y^3)
	\end{equation*}
	as $|x|  \le y \rightarrow 0+$. By the modular inversion formula for Dedekind eta function,
	we derive
	\begin{align*}
		\frac{(-q)_{\infty}}{(q)_{\infty}}=\sqrt{\frac{-i \tau}{2}}\frac{ \eta\left(-\frac{1}{2\tau}\right)}{\eta^2 \left(-\frac{1}{\tau}\right)}
		&= \sqrt{\frac{-i \tau}{2}} e^{\frac{\pi i}{8 \tau}}\left(1+ \mathcal{O}\left(e^{-2\pi\sqrt{N}}\right)\right).
	\end{align*}
	As $N \rightarrow \infty$ with $y=\dfrac{1}{4 \sqrt{N}}$ and $|x|\leq y$, we  arrive at
	\begin{align*}
		F_k(q)&=\frac{(-q)_{\infty}}{(q)_{\infty}}\left[2 f_{0,2,(4k-2)}(\tau)- f_{0,1,(4k-1)}(\tau)- f_{0,1,(4k-3)}(\tau)\right]\\
		&=\sqrt{\frac{-i \tau}{2}} e^{\frac{\pi i}{8 \tau}}\left(1+ \mathcal{O}\left(e^{-2\pi\sqrt{N}}\right)\right)\left[\frac{2k-1}{8}(-2 \pi i \tau)^2 + \mathcal{O}(N^{-3/2})\right]\\
		&=\left(\frac{2k-1}{8}\right)\frac{e^{\frac{\pi i}{8 \tau}}}{2\sqrt{\pi}}(-2 \pi i \tau)^{5/2} + \mathcal{O}\left(N^{-7/4}e^{\frac{\pi}{2}\sqrt{N}}\right).
	\end{align*}
\end{proof}
The following proposition describes the  behavior of $F_k(q)$ away from $q=1$.
\begin{prop}
\label{sum away1}
	For $y=\dfrac{1}{4\sqrt{N}}$ with $N>0$ and $y \leq |x| \leq 1/2$, 
	\begin{equation*}
		 |F_k(q)| \ll N^{1/4} e^{\frac{\pi}{2}\sqrt{N}-\frac{2-\sqrt{2}}{\pi} \sqrt{N}}.
	\end{equation*}
\end{prop}
\begin{proof}
	For $a>0$ and $a+b>0$, we observe
	\begin{equation*}
		\left|f_{0,a,b}(\tau)\right| \leq \sum_{n=1}^{\infty}|q|^{n/2} \ll N^{1/2}
	\end{equation*}
	by bounding each term.  As Lemma 3.5 in \cite{bd}, the inversion formula will give the bound
	\begin{align*}
		\left| \log \frac{(-q)_{\infty}}{(q)_{\infty}} \right| &\leq 2 \sum_{m=1}^{\infty} \frac{|q|^{2m-1}}{(2m-1)|1-q^{2m-1}|}\\
		&\leq 2 \sum_{m=1}^{\infty} \frac{|q|^{2m-1}}{(2m-1) \left(1-|q^{2m-1}|\right)} +\frac{2|q|}{|1-q|} -\frac{2|q|}{1-|q|} \\
		&= 2 \sum_{m=1}^{\infty} \frac{|q|^{m}}{m \left(1-|q^{m}|\right)} - \sum_{m=1}^{\infty} \frac{|q|^{2m}}{m \left(1-|q^{2m}|\right)} +2|q| \left( \frac{1}{|1-q|} -\frac{1}{1-|q|} \right) \\
		&= 2 \log \frac{1}{(|q|)_{\infty}} - \log \frac{1}{(|q|^2)_{\infty}} +2|q| \left( \frac{1}{|1-q|} -\frac{1}{1-|q|} \right) \\
		&\leq 2 \left( \frac{\pi}{12y}+\frac{1}{2} \log y+ \mathcal{O}(y) \right) -\left( \frac{\pi}{24y}+\frac{1}{2} \log 2y+ \mathcal{O}(y) \right) -\frac{1}{2 \pi y} \left( 1- \frac{1}{\sqrt{2}}\right) +\mathcal{O}(1).
	\end{align*}
	Thus we have
	\begin{equation*}
		\left| \frac{(-q)_{\infty}}{(q)_{\infty}} \right| \ll \sqrt{y} \exp\left[ \frac{\pi}{8y}-\frac{1}{2 \pi y}\left( 1-\frac{1}{\sqrt{2}} \right) \right].
	\end{equation*}
\end{proof}

Now we are ready to apply the circle method. By Cauchy's integral formula, we see that
\begin{align*}
	A_k(N)-B_k(N)&=\frac{1}{2 \pi i} \int_{\mathcal{C}} \frac{F_k(q)}{q^{N+1}} \, dq\\
	&=\int_{-1/2}^{1/2} F_k\left( e^{-\frac{\pi}{2\sqrt{N}}+2\pi i x} \right)e^{\frac{\pi}{2}\sqrt{N}-2\pi i N x} \, dx,
\end{align*}
where $\mathcal{C}=\left\{|q|=e^{-\frac{\pi}{2\sqrt{N}}}\right\}$. We divide the integral into the two parts:
\begin{align*}
	I'&=\int_{|x|\leq \frac{1}{4\sqrt{N}}} F_k\left( e^{-\frac{\pi}{2\sqrt{N}}+2\pi i x} \right)e^{\frac{\pi}{2}\sqrt{N}-2\pi i N x} \, dx\\
\intertext{and}
	I''&=\int_{\frac{1}{4\sqrt{N}} \leq |x|\leq \frac{1}{2}} F_k\left( e^{-\frac{\pi}{2\sqrt{N}}+2\pi i x} \right)e^{\frac{\pi}{2}\sqrt{N}-2\pi i N x} \, dx.
\end{align*}
The integral $I'$ will contribute the main term of the asymptotic formula and $I''$ will be absorbed into the error term. We first calculate the major contribution, i.e. $I'$.

\begin{prop} \label{majorarc}
	As $N \rightarrow \infty$,
	\begin{equation*}
		I'= (2k-1) \frac{\pi^3}{128 \sqrt{2}}N^{-7/4}I_{-7/2}\left(\pi\sqrt{N}\right)+\mathcal{O}\left(N^{-9/4}e^{\pi \sqrt{N}}\right),
	\end{equation*}
	where $I_{s}(z)$ is the modified Bessel function of the first kind.
\end{prop}
\begin{proof}
	By replacing $u$ by $4\sqrt{N}x$ and Proposition \ref{Sum_near1}, we deduce that
	\begin{align*}
		I'=& \frac{1}{4\sqrt{N}}\int_{-1}^1 F_k \left( e^{\frac{\pi}{2\sqrt{N}}(-1+iu)} \right) e^{\frac{\pi}{2} \sqrt{N} (1-iu)} \, du\\
		=& \frac{1}{4\sqrt{N}}\int_{-1}^1 \left[\frac{2k-1}{16 \sqrt{\pi}} \left( \frac{\pi(1-iu)}{2\sqrt{N}} \right)^{5/2} e^{\frac{\pi}{2}\sqrt{N}\left(\frac{1}{1-iu}+(1-iu)\right)} \right.\\
		& \hspace{180pt} \left. +\mathcal{O}\left(N^{-7/4}e^{\pi\sqrt{N}}\right)\right] \, du\\
		=& (2k-1) \frac{\pi^3}{128 \sqrt{2}} N^{-7/4} P_{5/2} + \mathcal{O}\left(N^{-9/4}e^{\pi\sqrt{N}}\right),
	\end{align*}	
	where
	\begin{equation*}
		P_s := \frac{1}{2 \pi i} \int_{1-i}^{1+i} v^s e^{\frac{\pi}{2} \sqrt{N} \left(v+\frac{1}{v}\right)} \, dv.
	\end{equation*}
	The same argument as  Lemma 4.2 in \cite{bm2} implies
	\begin{equation*}
		P_s = I_{-s-1}(\pi \sqrt{N})+\mathcal{O}\left(e^{\frac{3\pi}{4}\sqrt{N}}\right)
	\end{equation*}
	as $N \rightarrow \infty$, where we adopt Lemma 17 in \cite{w} and a bound for incomplete Gamma function.
\end{proof}

The contribution from $I''$ is relative small, so that it can be absorbed in the error term in $I'$.

\begin{prop} \label{minorarc}
	As $N \rightarrow \infty$,
	\begin{equation*}
		 |I''| \ll N^{1/4} e^{\pi\sqrt{N}-\frac{2-\sqrt{2}}{\pi} \sqrt{N}}.
	\end{equation*}
\end{prop}
\begin{proof}
	By Proposition \ref{sum away1}, we have
	\begin{equation*}
		|I''| \leq \int_{\frac{1}{4\sqrt{N}} \leq |x|\leq \frac{1}{2}} \left| F_k\left( e^{-\frac{\pi}{2\sqrt{N}}+2\pi i x} \right)e^{\frac{\pi \sqrt{N}}{2}-2\pi i N x} \right| \, dx \ll e^{\frac{\pi}{2}\sqrt{N}}  N^{1/4}e^{\frac{\pi}{2}\sqrt{N}-\frac{2-\sqrt{2}}{\pi} \sqrt{N}}.
	\end{equation*}
\end{proof}

By combining Propositions \ref{majorarc} and \ref{minorarc}, we obtain the asymptotic for $A_k (N) - B_k (N)$.

\begin{thm}\label{asym}
	For $k \geq 1$, as $N \rightarrow \infty$,
	\begin{equation*}
		A_k(N)-B_k(N)=(2k-1) \frac{\pi^3}{64 \sqrt{2}}N^{-7/4}I_{-7/2}\left(\pi\sqrt{N}\right)+\mathcal{O}\left(N^{-9/4}e^{\pi \sqrt{N}}\right).
	\end{equation*}
\end{thm}

By the famous asymptotic formula for the modified Bessel function of the first kind, $I_{s} (z) \sim \frac{e^z}{\sqrt{2\pi z}}$, Theorem \ref{mainthm} follows immediately from Theorem \ref{asym}.

Theorem \ref{second main} can be proved similarly, thus we will not give all details.  For $m \geq 2$, let
\begin{equation*}
	H_m (q) := \sum_{N=0}^{\infty} (C_1 (N) - m C_m (N))q^N =\frac{1}{(q)_{\infty}}\left( h(q)-m h(q^m) \right).
\end{equation*}
Next two propositions give the asymptotic behavior of $H_m(q)$ near $q=1$ and away from $q=1$.
\begin{prop}\label{H_near1}
	Assume $y=\dfrac{1}{2 \sqrt{6N}}$ and $|x|\leq y$. As $N \rightarrow \infty$,
	\begin{equation*}
		H_m(q)=\frac{m-1}{8\sqrt{2 \pi}}(-2 \pi i \tau)^{1/2}  e^{\frac{\pi i}{12 \tau}} + \mathcal{O}\left(N^{-3/4}e^{\pi \sqrt{\frac{N}{6}}} \right).
	\end{equation*}
\end{prop}
\begin{proof}
	From the partial fraction expansion
	\begin{equation*}
		\frac{\pi}{\sin \pi w}=\frac{1}{w}+\sum_{k=1}^{\infty} \frac{(-1)^k 2 w}{w^2-k^2},
	\end{equation*} 
	for $w \in \mathbb{C}$, $h(q)$ can be rewritten as
	\begin{align*}
		h(q) &=  \sum_{n=1}^{\infty} \frac{(-1)^{n+1} q^{n^2/2}}{-2 \pi i} \left( \frac{1}{n \tau}+ \sum_{k=1}^{\infty} \frac{(-1)^k 2 n \tau}{n^2 \tau^2 - k^2} \right) \\
		&= -\frac{1}{2\pi i \tau}\sum_{n=1}^{\infty} (-1)^{n+1} n^{-1} q^{n^2/2} -  \frac{1}{\pi i} \sum_{n=1}^{\infty} (-1)^{n+1} q^{n^2/2} \sum_{k=1}^{\infty} \frac{(-1)^k n \tau}{n^2 \tau^2-k^2} .
	\end{align*}
	Here, the first sum equals $\dfrac{1}{2 \pi i \tau}f_{1,1,0}(\tau)$. For the second sum $S_2$, we first note that
	\[
	|n^2\tau^2 -k^2| \geq | \textrm{Re} (n^2\tau^2 -k^2)| =  k^2+n^2(y^2-x^2) \ge k^2 .
	\]
	Therefore, we find that
	\[
	\sum_{k=1}^{ \infty } \left| \frac{1 }{n^2 {\tau}^2-k^2} + \frac{1}{k^2}  \right| \le \sum_{k=1}^{\infty}  \frac{2n^2 y^2}{k^4} \le 3 n^2 y^2 .
	\]
	 Hence, the second sum is
\begin{align*}
	S_2 &= \frac{1}{\pi i} \sum_{n=1}^{\infty} (-1)^{n+1} q^{n^2/2} \sum_{k=1}^{\infty}  (-1)^k  n\tau \left( \frac{1}{n^2 \tau^2 - k^2} + \frac{1}{k^2} - \frac{1}{k^2} \right)  \\
	&=\frac{1}{ \pi i} \sum_{n=1}^{\infty} (-1)^{n+1} n\tau q^{n^2/2}  \left(  \frac{\pi^2}{12}  +  T_{n, \tau}  \right),
	\end{align*}
	where $|T_{n,\tau}| \le  3 n^2 y^2$.  By employing
	\[
	\sum_{n=1}^{\infty} n^{3} e^{-\pi n^2 y } = \mathcal{O} ( y^{-2} )
	\]
	and $f_{-1,1,0}(\tau) = -1/4 + \mathcal{O} (y)$  from the same argument as Lemma \ref{f_0}, we can deduce that
	\[
	 S_2  =  \mathcal O (y)
	\]
	as $|x| \le y \rightarrow 0+$. Similarly, we can find
	\begin{equation*}
		h(q^m)=\frac{1}{2 \pi i m \tau} f_{1,m,0}(\tau) + \mathcal{O}\left( y \right).
	\end{equation*}
	Now we investigate the asymptote for $f_{1,1,0}(\tau) - f_{1,m,0} (\tau)$ when $|x|\leq y$ and $y \rightarrow 0+$.  By the same argument as Lemma \ref{f_0}, we find that
	\begin{equation*}
		f_{1,1,0} (\tau) - f_{1,m,0}(\tau)= \frac{m-1}{8}(2\pi i \tau)+\mathcal{O}(y^2)
	\end{equation*}
	as $y \rightarrow 0+$ with $|x|\leq y$. In summary, we have arrived at
	\begin{align*}
		h(q)-m h(q^m)
		&= \frac{m-1}{8} + \mathcal{O}({y}),
	\end{align*}
	which completes the proof by combining with the expansion for $1 / (q)_{\infty} $ \cite[eqn. (3.8)]{bm2}.
\end{proof}

\begin{prop}\label{H_away1}
	If $y=\dfrac{1}{2\sqrt{6N}}$ with $N>0$ and $y \leq |x| \leq 1/2$,
	\begin{equation*}
		 \left|\frac{1}{(q)_{\infty}}\left( h(q)-m h(q) \right)\right| \ll N^{1/2} e^{\pi\sqrt{\frac{N}{6}} -\frac{\sqrt{3N} (\sqrt{2} -1) }{\pi}}    .
	\end{equation*}
\end{prop}

By employing the Hardy-Ramanujan circle method on the integral representation for the coefficients of $H_m(q)$ with the contour  $\mathcal{C}=\left\{ |q|=e^{-\frac{\pi}{\sqrt{6 N}}} \right\}$, we find the asymptotic formula with aids of Propositions \ref{H_near1} and \ref{H_away1}. Theorem \ref{second main} follows immediately from this asymptotic formula.
\begin{thm}
	For $m \geq 2$, as $N \rightarrow \infty$,
	\begin{align*}
		&C_1 (N) - m C_m (N)\\
		&=(m-1) \frac{\pi}{16 \sqrt[4]{54}} N^{-3/4} I_{-3/2} \left( \pi\sqrt{\frac{2N}{3}}\right)+\mathcal{O}\left(N^{-5/4}e^{\pi\sqrt{\frac{2N}{3}}} \right).
	\end{align*}
\end{thm}

\section{Concluding Remarks}

Though we have obtained that
\[
A_k (n) > B_k (n)
\]
holds for large enough $n$'s, it is still desirable to find a combinatorial object which is identical with $A_k (n) - B_k (n)$. The existence of such object will guarantee that $A_k (n) \ge B_k (n)$. Moreover, as $\oospt(n)$ is the sum of the differences between $A_k (n)$ and $B_k (n)$, such a combinatorial model can be regarded as a refinement of Andrews, Chan, Kim, and Osburn's result.

 Numerical data suggest that for all positive integers $n$,
 \[
 \oospt(n) > \frac{1}{8} \overline{p} (n)
 \]
 though $\oospt(n) \sim \frac{1}{8} \overline{p} (n)$. It would be very interesting if one can find a combinatorial or a $q$-theoretic proof for the above inequality.

\section*{Acknowledgments}
The authors  thank Robert Rhoades and the referee for valuable comments on an earlier version of this paper.


\begin{thebibliography}{15}

\bibitem{Andrews:spt}
G.E. Andrews, \emph{The number of smallest parts in the partitions of $n$}, J. Reine Angew. Math. \textbf{624} (2008), 133--142.

\bibitem{ack}
G.E. Andrews, S.H. Chan, B. Kim, \emph{The odd moments of ranks and cranks}, J. Combin. Theory Ser. A \textbf{120} (2013), no. 1, 77--91.

\bibitem{ag}
G.E. Andrews, F.G. Garvan, \emph{Dyson's crank of a partition}, Bull. Amer. Math. Soc. (N.S.) \textbf{18} (1988), no. 2, 167--171.

\bibitem{acko}
G.E. Andrws, S.H. Chan, B. Kim, R. Osburn, \emph{The first positive rank and crank moments for overpartitions}, Ann. Comb. \textbf{20} (2016), no. 2, 193-–207.

\bibitem{atg}
A.O.L. Atkin, F.G. Garvan, \emph{Relations between the ranks and cranks of partitions}, Ramanujan J. \textbf{7} (2003), no. 1-3, 343--366.

\bibitem{blo}
K. Bringmann, J. Lovejoy, R. Osburn, \emph{Rank and crank moments for overpartitions}, J. Number Theory \textbf{129} (2009), no. 7, 1758--1772.

\bibitem{bm2}
K. Bringmann, K. Mahlburg, \emph{Asymptotic inequalities for positive crank and rank moments}, Trans. Amer. Math. Soc. \textbf{366} (2014), no. 2, 1073-–1094 .

\bibitem{bd}
K. Bringmann, J. Dousse, \emph{On {D}yson's crank conjecture and the uniform asymptotic behavior of certain inverse theta functions}, Trans. Amer. Math. Soc. \textbf{368} (2016), no. 5, 3141--3155.


\bibitem{dyson}
F.J. Dyson, \emph{Some guesses in the theory of partitions}, Eureka \textbf{8} (1944), 10--15.

\bibitem{gar2}
F. G. Garvan, \emph{Higher order spt-functions}, Adv. Math. \textbf{228} (2011), no. 1, 241--265.

\bibitem{kk}
B. Kim, E. Kim, \emph{On the subpartitions of the ordinary partitions, II}, Electron. J. Combin. \textbf{21} (2014), no. 4, Paper 4.21. 

\bibitem{lc}
J. Lovejoy, S. Corteel, \emph{Overpartitions}, Trans. Amer. Math. Soc. \textbf{356} (2004), no. 4, 1623--1635.

\bibitem{w}
E. Wright, \emph{Asymptotic Partition Formulae: (II) Weighted Partitions}, Proc. London Math. Soc. (2) \textbf{36} (1933), no. 1, 117--141

\bibitem{z}
  D. Zagier, \emph{The Mellin transform and other useful analytic techniques}, Appendix to E. Zeidler, Quantum Field
Theory I: Basics in Mathematics and Physics. A Bridge Between Mathematicians and Physicists Springer-Verlag,
Berlin-Heidelberg-New York (2006), 305--323.

\end{thebibliography}
\end{document}